\documentclass[a4paper]{amsart}
\usepackage{graphicx,url,amsthm,xspace,enumitem,subcaption}
\usepackage[dvipsnames]{xcolor}
\usepackage{amscd}
\usepackage{multirow}
\usepackage{mathtools}
\usepackage{xy}
\xyoption{all}

\usepackage{pinlabel}
\usepackage[linktocpage]{hyperref}

\def\epsilon{\varepsilon}

\newcommand{\C}{\mathbb{C}}
\newcommand{\M}{\mathbb{M}}
\newcommand{\D}{D}
\newcommand{\F}{\Z/2\Z}
\newcommand{\R}{\mathbb{R}}
\newcommand{\Q}{\mathbb{Q}}
\newcommand{\Z}{\mathbb{Z}}
\newcommand{\ft}{\mathfrak{t}}
\newcommand{\fs}{\mathfrak{s}}

\newcommand{\es}{S^2\times S^1}
\newcommand{\ed}{D^3\times S^1}
\newcommand{\jacobi}[2]{\big(\frac{#1}{#2}\big)}

\DeclareMathOperator{\Ima}{Im}
\DeclareMathOperator{\rank}{rank}

\theoremstyle{plain}
\newtheorem{theorem}{Theorem}[section]

\newtheorem{lemma}[theorem]{Lemma}
\newtheorem{prop}[theorem]{Proposition}

\theoremstyle{definition}
\newtheorem{question}[theorem]{Question}
\newtheorem*{question*}{Question}
\newtheorem{definition}[theorem]{Definition}

\theoremstyle{remark}
\newtheorem{remark}[theorem]{Remark}

\title{Non-orientable slice surfaces and inscribed rectangles}

\author{Peter Feller}
\email{peter.feller@math.ethz.ch}
\address{ETH Zurich, Department of Mathematics, Zurich, Switzerland}

\author{Marco Golla}
\email{marco.golla@univ-nantes.fr}
\address{CNRS, Laboratoire de Math\'ematiques Jean Leray, Nantes, France}

\begin{document}
\textwidth=360pt \textheight=615pt
\maketitle
\begin{abstract}
We discuss differences between genera of smooth and locally-flat non-orientable surfaces in the 4--ball with boundary a given torus knot or 2--bridge knot. In particular, we establish that a result by Batson on the smooth non-orientable 4--genus of torus knots does not hold in the locally-flat category. We further show that certain families of torus knots are not the boundary of an embedded M\"obius band in the 4--ball and other 4--manifolds.

Our investigation of non-orientable surfaces with boundary a given torus knot is motivated by our approach to unify the proof of the existence of inscribed squares and of inscribed rectangles with aspect ratio $\sqrt3$ in Jordan curves with a regularity condition. This generalizes a result by Hugelmeyer for smooth Jordan curves.

\end{abstract}

% !TEX root = ../squarepeg.tex

\section{Introduction}

Let $\Gamma\subset \R^2$ be a Jordan curve; that is, $\Gamma$ is the image of an injective continuous function $S^1 \to \R^2$.
We say that $\Gamma$ is \emph{locally $1$--Lipschitz} if for each point $p \in \Gamma$ there is a neighbourhood $U$ of $p$ such that $\Gamma \cap U$ is the graph of a $1$--Lipschitz function. (See Definition~\ref{d:locally1lip} below.)
A Euclidean rectangle in $\R^2$ is said to be \emph{inscribed} in $\Gamma$ if its four corners belong to $\Gamma$.
The \emph{aspect ratio} of a rectangle whose sides have lengths $a$ and $b$ is $b/a$; note that this is only defined up to reciprocals.

\begin{theorem}\label{thm:main}
Let $\Gamma\subset \R^2$ be a locally $1$--Lipschitz Jordan curve.
Then, for all integers $n\geq 2$, there exists an integer $1\leq k \leq n-1$ such that
$\Gamma$ has an inscribed rectangle with aspect ratio $\tan\left(\frac{k\pi}{2n}\right)$.

In particular, $\Gamma$ has an inscribed square and an inscribed rectangle with aspect ratio $\sqrt{3}$.
\end{theorem}

In fact, as we will see below, we will give a condition on the curve $\Gamma$ (see Proposition~\ref{prop:reductiontonoMoebiusinS1xB3}) for which the conclusion of Theorem~\ref{thm:main} holds, and we will prove that locally $1$--Lipschitz Jordan curves satisfy this condition.
At this point, we do not know of any curve $\Gamma$ for which our scheme of proof does not apply; see Remark~\ref{rem:canonealwaysfindN}. For instance, while not all polygonal Jordan curves are locally $1$--Lipschitz, they satisfy the same condition, and therefore the statement can be extended to polygonal curves.

The condition on $\Gamma$ is chosen such that the proof of the theorem reduces to the following result.

\begin{theorem}\label{thm:nomob}
Let $K \subset \es$ be the torus knot $K_{1,2n}$ in $\es$.
If $n$ is not a square, then $K$ does not bound a locally-flat M\"obius band in $\ed$.
\end{theorem}
Here, for coprime integers $p$ and $q$, and identifying $S^2=\C\cup\{\infty\}$, the \emph{torus knot} $K_{p,q}$ in $\es$ is
$\left\{\left(e^{2\pi i p t},e^{2\pi i q t}\right)\;\middle|\;t\in\R\right\}\subset\es$.
Our proof of Theorem~\ref{thm:nomob} combines a branched double cover construction and a very simple intersection form obstruction. While our proof does not work when $n$ is a square, we believe that the result also holds without that assumption. % as well.

Theorem~\ref{thm:main} is motivated by a question posed by Toeplitz in 1911, often referred to as the \emph{square peg problem}: does every Jordan curve contain an inscribed square?
While the question remains open in full generality, it has been resolved in the positive for many classes of curves. We refer to Matschke's excellent survey on the topic~\cite{Matschke14}. See also Schwartz~\cite{Schwartz} and Tao~\cite{Tao} for further progress.

%\MG{Small edits to make line breaks better.}
The starting point of the present article was a beautiful idea of Hugelmeyer~\cite{Hugelmeyer18}, which is the first article to address Toeplitz's question for rectangles of a fixed aspect ratio $r>1$. In the paper, Hugelmeyer establishes Theorem~\ref{thm:main} for $n\geq 3$ (the case $n=2$ was previously known) and smooth Jordan curves, using an idea of proof in line with Vaughan's proof~\cite{Meyerson81,Matschke14} of the following result.
Every Jordan curve has an inscribed rectangle. In his surprising proof, Vaughan elegantly reduces this result to the statement that there is no proper embedding of the M\"obius band into the upper-half space $\R^2\times [0,\infty)$.

In more detail, in case $\Gamma$ is smooth, Hugelmeyer gives a reduction of Theorem~\ref{thm:main} to the smooth analog of Theorem~\ref{thm:nomob}. He then observes that, in the smooth setting and for $n\geq 3$, Theorem~\ref{thm:nomob} follows from a result by Batson~\cite{Batson}: for $n\geq 3$, the torus knot $T_{2n-1,2n}$ in the 3--sphere $S^3=\partial B^4$ does not arise as the boundary of a smoothly embedded M\"obius band in the 4--ball $B^4$.

The need for locally-flat statements in our approach, in contrast to Batson's smooth results, as well as the differences between the smooth and topological \emph{orientable} 4-genus~\cite{Rudolph_84_SomeTopLocFlatSurf,BaaderFellerLewarkLiechti_15}, motivated us to compare locally-flatly and smoothly embedded \emph{non-orientable} surfaces in the 4--ball.

For a knot $K$ in $S^3$, let $\gamma_4(K)$ denote the smallest first Betti number among smooth non-orientable surfaces in $B^4$ with boundary $K\subset S^3=\partial B^4$.
Similarly, denote by $\gamma_4^\text{top}(K)$ the smallest first Betti number among locally-flat non-orientable surfaces in $B^4$ with boundary $K\subset S^3=\partial B^4$.
\begin{theorem}\label{thm:gammafortorusknots}
For each integer $n\geq 5$, we have $\gamma_4^{\rm top}(T_{2n-1,2n})<\gamma_4(T_{2n-1,2n})$.
\end{theorem}
Batson proved in~\cite{Batson} that $\gamma_4(T_{2n-1,2n})=n-1$ for all integers $n>1$.
We do not determine $\gamma_4^{\rm top}(T_{2n-1,2n})$, but only show that it is at most $n-2$. We believe that the following two questions are open.
Does $T_{9,10}$ bound a locally-flat punctured Klein bottle? Are there torus knots that bound a locally-flat M\"obius band in $B^4$, but not a smooth one?

In general we study which torus knots in $S^3$ arise as the boundary of a locally-flat M\"obius band in~$B^4$.
In particular in Section~\ref{subsec:mob}, we focus on the torus knots $T_{2n,2n+1}$: in contrast to $T_{2n-1,2n}$, it is in general unknown which ones bound smooth M\"obius bands in~$B^4$.
%; see Section~\ref{subsec:mob}.

%\MG{Changed the proposition and the preceding sentence. Old version commented.}
For example, we generalize the folklore result that neither $T_{4,5}$ nor $T_{5,6}\subset S^3$ bounds a locally-flat M\"obius band in~$B^4$.

\begin{prop}\label{p:T45}
Let $p\equiv 5 \pmod 8$ be a positive integer. Then $T_{p,p\pm 1}$ does not bound a locally-flat M\"obius band in $B^4$.
\end{prop}

%For example, we generalize the folklore result that $T_{4,5}\subset S^3$ does not bounds a M\"obius band in~$B^4$.
%
%\begin{prop}\label{p:T45}
%Let $n\equiv 2 \pmod 4$ be a positive integer such that $2n+1$ is a prime. Then $T_{2n,2n+1}$ does not bound a locally-flat M\"obius band in $B^4$.
%\end{prop}
%
Branched covers feature prominently in most topological proofs. We use them in combination with correction terms in Heegaard Floer homology and with Donaldson's diagonalisation theorem to close the paper with two results on torus knots and 2-bridge knots which cannot bound smooth M\"obius bands in $B^4$.

\subsection*{Addendum} Since the appearance of this manuscript, Greene and Lobb have used symplectic topology to extend Hugelmeyer's result from~\cite{Hugelmeyer18} to cyclic quadrilaterals inscribed in smooth curves~\cite{GreeneLobb-rect, GreeneLobb-cyclic}. Our main results on inscribed quadrilaterals, Theorem~\ref{thm:main} and Proposition~\ref{prop:reductiontonoMoebiusinS1xB3}, go in a different direction, since we aim at weakening the regularity assumption on the curve. Our results on non-orientable surfaces in 4-manifolds are completely independent of their work.

\subsection*{Structure} In Section~\ref{s:keyobs} we prove Proposition~\ref{prop:reductiontonoMoebiusinS1xB3}, which contains the key technical property of curves that allows us to link metric geometry to topology. In Section~\ref{s:metric}, we prove Theorem~\ref{thm:main}, assuming Theorem~\ref{thm:nomob}, which is proved in Section~\ref{s:nomob}. In Section~\ref{s:balls} we prove Theorem~\ref{thm:gammafortorusknots}, along with a number of obstructions to the existence of smooth and locally-flat M\"obius bands in $B^4$ bounding torus knots and $2$--bridge knots in $S^3$.

\subsection*{Acknowledgements}
%\MG{Added Nantes, ETH, MPIM. Initialised both of us. Please rephrase nicely :)}\PF{Dont see any renaming needed. Like the initials thing.}
PF thanks Luca Studer for introducing him to Vaughan's beautiful proof that every Jordan curve has an inscribed square and the survey by Matschke. We thank Andr\'{a}s Stipsicz for inquiring about locally-flat non-orientable surfaces filling torus knots. PF gratefully acknowledges support by the Swiss National Science Foundation Grant 181199. This project started when PF visited the University of Nantes, and was partially carried on while MG visited ETH and when both authors stayed at MPIM. We thank all three institutions for their support.
% !TEX root = ../squarepeg.tex

\section{From curves to M\"obius bands}\label{s:keyobs}

In this section, we fix a continuous injection $\alpha\colon S^1\to \C$ and denote the image of $\alpha$ (by definition, a Jordan curve) as $\Gamma\subset \C$.

Let $\M$ denote $S^1\times S^1/(\Z/2\Z)$, that is, the quotient of $S^1\times S^1$ by the relation $(x,y)\sim (y,x)$. Note that $\M$ is a M\"obius band.
For each positive integer $n$, we consider the map
\[
\Psi_n\colon \M\to \C^2, \{s,t\}\mapsto \left(\frac{\alpha(s)+\alpha(t)}{2},\left(\alpha(s)-\alpha(t)\right)^{2n}\right).
\]
Note that the image of $\Psi_n$ does not depend on the parametrization $\alpha$ of $\Gamma$.

\begin{remark}\label{rem:keyobs}
The maps $\Psi_n$ were studied by Hugelmeyer~\cite{Hugelmeyer18}, extending an idea of Vaughan; compare~\cite{Meyerson81,Matschke14}. A crucial observation is the following. For $n\geq 2$, the map $\Psi_n$ is injective if and only if there exists an integer $1\leq k \leq n-1$ such that
$\Gamma$ has an inscribed rectangle with aspect ratio $\tan\left(\frac{k\pi}{2n}\right)$.
\end{remark}

\subsection{The image of $\Psi_1$ is topologically locally flat} %We briefly recall the notions of local-flatness and (local) transversality. %We refer to~\cite[Section 9]{FreedmanQuinn_90_TopOf4Manifolds} for details.
Denote with $O_{k}$ the origin of $\R^k$.
\begin{definition}
Fix integers $m>n>0$. A subset $F$ of an $m$--manifold $M$ is called \emph{locally-flat (of dimension $n$)} at $x\in F$ if there exist an open neighbourhood $U\subset M$ of $x$ such that $(U,U\cap F)$ is homeomorphic to an open subset of $(\R^m, \R_{\ge 0}\times \R^{n-1} \times \{O_{m-n}\})$.
The subset $F$ is called \emph{locally-flat} if it is locally-flat for all $x\in F$.

Two locally-flat submanifolds $F, F'$ without boundary, of dimension $n$ and $n'$ respectively, are said to intersect \emph{transversely} in $M$ if: either $F \cap F'$ is empty, or $n+n' \ge m$ and every point $x \in F \cap F'$ has a neighbourhood $U \subset M$ such that $(U,U \cap F,U \cap F')$ is homeomorphic to an open subset of $(\R^m, \R^{n}\times \{O_{m-n}\}, \{O_{m-n'}\}\times\R^{n'}).$
\end{definition}

\begin{lemma}\label{lem:ImPsi_1islocflat}
$\Psi_1$ is injective and its image $M\coloneqq\Psi_1\left(\M\right)\subset \C^2$ is a locally-flat M\"obius band. Furthermore, there is a regular neighbourhood $N$ of $\C\times\{0\}$ such that $\partial N$ intersects $M$ transversely and the pair $(N,N\cap M)$ is homeomorphic to
\[
\left(\C\times \D_1,\{(x,y)\in S^1\times \D \mid |y|x^2=y\}\right) = \left(\C\times \D_1, \{(s, r(s^2)) \mid r\in[0,1], s \in S^1\}\right).
\]
In particular, $(\partial N,M\cap \partial N)$ is homeomorphic to $(\C\times S^1,K_{1,2})$.
\end{lemma}

Here, $\D_r\coloneqq \{z\in \C\;|\; |z|\leq r\}$ is the unite disc of radius $r$, and a closed regular neighborhood is understood to be a subset $N\subset \C^2$ such that there exist a homeomorphism $\Psi\colon \C^2\to\C^2$ with $\Psi(N)=\C\times \D_r$ and $\Psi$ restricts to the identity on $\C\times\{0\}$.
%, and$K_{1,2}=\left\{\left(e^{2\pi i t},e^{2\pi i 2t}\right)\;\middle|\;t\in\R\right\}$.
%, where $\D_r\coloneqq \{z\in \C\;|\; |z|\leq r\}$.

\begin{proof}
Clearly, $\Psi_1$ is injective. The rest of the statement is easy to check
%Local flatness is easy to check
when $\Gamma$ is the unit circle $\Gamma_{\rm std} = S^1\subset \C$. Indeed, the image $M_{\rm std}\coloneqq \Ima(\Psi_1)$ is
a smooth 2-submanifold of $\C^2$; in particular, the Möbius band $M_{\rm std}$ is locally-flat. And, for $N=\C\times \D_1$, $\partial N=\C\times S^1$ intersects $M_{\rm std}$ transversely in
\[\left\{\left(\tfrac{\sqrt{3}}{2}e^{2\pi i t},-e^{2\pi i 2t}\right)\;\middle|\;t\in\R\right\}\subset \frac{\sqrt{3}}{2}S^1\times S^1\subset\C\times S^1.\]

For the general case, let $\phi\colon \C\to \C$ be a compactly supported homeomorphism such that $\Gamma = \phi(\Gamma_{\rm std})$, which exists by the Jordan--Sch\"onflies theorem.
The statement follows by identifying $\C^2$ with the space of unordered pairs $\C^2/(\Z/2\Z)=\left\{\{x,y\}\;\middle|\;x,y\in\C\right\}$ and observing that the self-homeomorphism of $\C^2/(\Z/2\Z)$ given by $\Phi(\{x,y\})=\{\phi(x),\phi(y)\}$ induces a homeomorphisms of topological pairs between $(\C^2,M_{\rm std})$ and $(\C^2,M)$.
To be explicit, $(\C^2,M_{\rm std})$ and $(\C^2,M)$ are homeomorphic as pairs via %the map
\begin{align*}
\Psi\colon \C^2&\to\C^2\\
(z,w)&\mapsto \left(\frac{\phi\left(z\pm \frac{\sqrt{w}}{2}\right)+\phi\left(z\mp \frac{\sqrt{w}}{2}\right)}{2},\left(\phi\left({\textstyle z\pm \frac{\sqrt{w}}{2}}\right)-\phi\left({\textstyle z\mp \frac{\sqrt{w}}{2}}\right)\right)^2\right). \qedhere
\end{align*}
\end{proof}

\subsection{The image of $\Psi_1$ under taking powers in the second coordinate}
Note that $\Psi_n=p_n\circ\Psi_1$, where $p_n:\C^2\to\C^2, (z,w)\mapsto (z,w^n)$.

\begin{lemma}\label{lem:imPsi_nisflat}
Fix an $n\geq 2$. If $\Psi_n$ is an injection on an open subsurface $S$ of $\M\setminus\partial \M$, then $\Psi_n(S)\subset \C^2$ is a locally-flat surface.
\end{lemma}

\begin{proof}
This is immediate from Lemma~\ref{lem:ImPsi_1islocflat}.
For $(z,w)\in \Psi_n(S)$, we have $w\neq 0$. Let $(z,u)$ be a preimage of $(z,w)$ under $p_n$ that lies in $M\subset \C^2$. Let $U$ be an open neighborhood of $(z,u)$ that witnesses the local flatness of $M$ at $(z,u)$ and that maps injectively under $p_n$.
Furthermore, we choose $U$ sufficiently small such that $p_n(U\cap M)=p(U)\cap \Psi_n(S)$. Thus, $p(U)$ is an open neighborhood of $(z,w)$ that establishes the local flatness of $M_n$ at $(z,w)$.
\end{proof}

Using Lemma~\ref{lem:imPsi_nisflat} and Theorem~\ref{thm:nomob}, we can obtain the following.

\begin{prop}\label{prop:reductiontonoMoebiusinS1xB3}
Fix a Jordan curve $\Gamma \subset \C$, an integer $n\geq 2$, and a positive real number $d$.
Suppose that $N$ is a regular closed neighborhood of $\C\times \{0\}$ in $\C^2$ containing $\C\times \D_{d^{2n}}$ such that:
\begin{itemize}
\item the map $\Psi_n$ and $\partial N$ are transverse;
\item the pair $(\partial N, \Ima(\Psi_n)\cap\partial N)$ is homeomorphic to $(\C\times S^1,K_{1,\pm2n})$.
\end{itemize}
Then, the there exists $a,b\in\M$ such that $\Psi_n(a)=\Psi_n(b)\in \C \times (\C\setminus \D_{d^{2n}}^\circ)$; in other words, there exists an integer $1\leq k \leq n-1$ such that $\Gamma$ has an inscribed rectangle with aspect ratio $\tan\left(\frac{k\pi}{2n}\right)$ and diameter larger than $d$.
\end{prop}
Here, $\Psi_n$ and $\partial N$ being \emph{transverse} is defined to mean that every point $x\in \Psi_n^{-1}(\partial N)$ has an open neighbourhood $U_x$ such that $\Psi_n$ restricts to an injection on $U_x$ with image $\Psi_n(U_x)$ a locally-flat surface that
intersects $\partial N$ transversely.
\begin{proof}
It suffices to establish the theorem for $n$ prime. Indeed, if $p$ is a prime factor of $n$, then the non-injectivity of $\Psi_p$ implies that of $\Psi_n$, since $\Psi_n$ arises as the concatenation of a map with $\Psi_p$. So, from here on, we only consider $n$ prime.

Set $X\coloneqq \C^2\setminus N^\circ$. Assume towards a contradiction that $\Gamma$ has no inscribed rectangle with aspect ratio $\tan\left(\frac{k\pi}{2n}\right)$ of diameter larger than $d$.
Equivalently, as in Remark~\ref{rem:keyobs}, $\Psi_n$ is injective restricted to $\Psi_n^{-1}(\C\times(\C\setminus {D_{d^{2n}}^\circ}))$.
Thus, by Lemma~\ref{lem:imPsi_nisflat}, $M_n\coloneqq \Ima(\Psi_n)\cap X$ is locally-flat in its interior and, by transversality of $\Psi_n$ and $\partial N$, also at its boundary. Hence, $M_n$ is a locally-flat properly embedded surface.
It is either a closed disc or a Möbius band, since it is homeomorphic to a closed subsurface of $\M$ with connected boundary.
Since $\partial M_n=\partial N \cap \Ima(\Psi_n)$ is a circle that is not null-homotopic in $\partial X=\partial N$, $M_n$ must be a Möbius band.
In conclusion, $M_n$ is a locally-flat M\"obius band, properly embedded in $X$, such that the pair $(\partial X, \partial M_n)$ is homeomorphic to $(\C\times S^1,K_{1,2n})$, which yields a contradiction since the image of $M_n$ under an embedding $(X, \partial X)\hookrightarrow (\ed,\es)$ is a locally-flat Möbius band in $\ed$ that cannot exist by Theorem~\ref{thm:nomob}.
\end{proof}

\begin{remark}\label{rem:canonealwaysfindN}
%Our strategy to prove Theorem~\ref{thm:main} will be to prove that locally 1--Lipschitz , and to show that the conditions of Proposition~\ref{prop:reductiontonoMoebiusinS1xB3} are satisfied.
We observe (as in Remark~\ref{rem:keyobs}) that, for a Jordan curve $\Gamma \subset \C$ the following are equivalent:
\begin{itemize}
  \item $\Psi_n$ is an injection in a neighborhood of the boundary of the M\"obius band $\M$;
  \item\label{nosmallrect} there exists an $\epsilon>0$ such that $\Gamma$ has no inscribed rectangles with the aspect ratios $\tan\left(\frac{k\pi}{2n}\right)$ and diameter less than $\epsilon$.
\end{itemize}

Now, if $\Gamma$ is a Jordan curve with no `small' rectangles with the aspect ratios $\tan\left(\frac{k\pi}{2n}\right)$ as above and $A$ a neighbourhood on which $\Psi_n$ is injective, then $\Psi_n(A\setminus\partial A)$ is a locally-flat surface by Lemma~\ref{lem:imPsi_nisflat}. Any small regular neighborhood of $N$ of $\C^2\times\{0\}$ can be made to have its boundary transversal to $\Psi_n(A\setminus\partial A)$ by a small compactly supported ambient isotopy; see~\cite[Section~9.5]{FreedmanQuinn_90_TopOf4Manifolds}.
Thus, the first condition in Proposition~\ref{prop:reductiontonoMoebiusinS1xB3} above is automatically satisfied for such $\Gamma$.
The interesting question is thus the following:
\begin{question}
Can $N$ be found so that the 1--manifold $A\cap \partial N$ is connected of knot type $K_{1,\pm2n}$?
\end{question}
If the answer to the question were positive, by Proposition~\ref{prop:reductiontonoMoebiusinS1xB3}, $\Gamma$ has inscribed rectangles with the aspect ratios $\tan\left(\frac{k\pi}{n}\right)$.
It seems conceivable that the answer is always yes.
We take this as an indication that at least in principle, the general case of $\Gamma$ is in the realm of being treated using this approach, and the difficulty lies in under standing the knot (or link) type of $A\cap \partial N$ as a knot (link) in $\partial N\cong \C\times S^1$. And so, we understand that application to 1--Lipschitz curves given in the next section as proof of concept for this method, rather than its optimal use. 
\end{remark}

In the next section, we will study a family of Jordan curves for which we are able to establish the conditions from Proposition~\ref{prop:reductiontonoMoebiusinS1xB3}.

\section{Proof of Theorem~\ref{thm:main}}\label{s:metric}
\begin{definition}\label{d:locally1lip}
A Jordan curve $\Gamma \subset \C$ is said to be \emph{locally $1$--Lipschitz} if, for each $p\in \Gamma$, there exist $\epsilon > 0$ and an isometry $(-\epsilon,\epsilon) \times (-\epsilon,\epsilon) \to U \subset \C$ centered at $p$ such that, in these coordinates, $\Gamma \cap U$ is the graph of a $1$--Lipschitz function.
\end{definition}

Being locally $1$--Lipschitz is a condition on Jordan curves $\Gamma$ that implies that, for sufficiently small $r$, the neighborhood $N\coloneqq \C\times \D_{r^{2n}}$ is a neighborhood as required for Proposition~\ref{prop:reductiontonoMoebiusinS1xB3}. For this $N$, the projection of $\Psi_n(\M)\cap \partial N \subset \C\times\C$ to the first coordinate corresponds to the set $\Gamma_r\subset \C$ of all midpoints of pairs $(\gamma_1,\gamma_2)\in\Gamma\times \Gamma$ that are distance $r$ apart.
A key step of the proof below is the claim that $\Gamma_r$ is a Jordan curve in case $\Gamma$ is locally $1$--Lipschitz. The point is that this implies that $K\coloneqq\Psi_n(\M)\cap \partial N$ is a subset of the \emph{torus} $\Gamma_r\times S^1_{r^{2n}}$, so, given we know $K$ is a connected $1$--manifold in $\partial N=\C\times S^1_{r^{2n}}$, $K$ must be a torus knot. We start with the following elementary lemma.

\begin{lemma}\label{lem:etaunique}
Let $g\colon\R\to\R$ be $1$--Lipschitz. For every $t\in \R$ and $r\geq 0$, there exists a unique $\eta_r(t)\in[0,r/2]$ such that $|t+\eta_r(t)+ig(t+\eta_r(t))-(t-\eta_r(t)+ig(t-\eta_r(t)))|=r$.
Moreover, the map $(r,t) \mapsto \eta_r(t)$ is continuous.
\end{lemma}

\begin{proof}
For existence and uniqueness, it is enough to prove the statement for $t=0$.
The function $|x+ig(x))-(-x+ig(-x))|^2=(g(x)-g(-x))^2+(2x)^2$ is strictly increasing on $[0,\infty)$. Hence, there exists a unique $\eta$ with $(g(\eta)-g(-\eta))^2+(2\eta)^2=r^2$ for every $r\in[0,\infty)$.
%Continuity in $r$ is obvious from the argument above;
Continuity in $(r,t)$ is easy to note.
\end{proof}

\begin{proof}[Proof of Theorem~\ref{thm:main}]
Fix a parametrization $\alpha\colon S^1\to \C$ of a locally $1$--Lipschitz curve $\Gamma$. We aim to show that, for all $n$, the map $\Psi_n$ is not injective.
By Proposition~\ref{prop:reductiontonoMoebiusinS1xB3},
it suffices to show the following.
\begin{enumerate}
  \item\label{item:1} $\Psi_n$ is an injection restricted to an open neighborhood $A$ of the boundary of $\M$ and no $m\in \M\setminus A$ gets mapped to $\Psi_n(A)$,
  \item\label{item:2} there exists a regular neighborhood $N$ of $\C\times\{0\}$ such that $\partial N$ intersects $\Ima(\Psi_n)$ only in $\Psi_n(A)$, which is locally-flat by Lemma~\ref{lem:imPsi_nisflat}, and the intersection is transverse for all elements of $\Ima(\Psi_n)\cap\partial N=\Psi_n(A)\cap\partial N$, and
  \item\label{item:3} the pair $(\partial N, \partial N\cap\Ima(\Psi_n))$ is homeomorphic to $(\C\times S^1,K_{1,2n})$.
\end{enumerate}
Indeed, choosing $d$ such that $\C\times D_{d^{2n}}\subset N$ gives the assumptions of Proposition~\ref{prop:reductiontonoMoebiusinS1xB3}. The remainder of this proof is concerned with establishing~\eqref{item:1},~\eqref{item:2},~and \eqref{item:3}. We use the following.

Let $(x,y)$ be a point on $\Gamma$. By assumption, we have that in a small neighborhood $U$ of $(x,y)$, $\Gamma$ is given as the image of $\{L(t,f(t)) \mid t\in(-\frac32\epsilon,\frac32\epsilon)\}$, where $L$ is an isometry of $\R^2$ (with the Euclidean metric), $\epsilon>0$, and $f\colon (-\frac32\epsilon,\frac32\epsilon)\to\R, 0\mapsto 0$ is 1--Lipschitz. In fact, we arrange that $L$ satisfies
$L(U_\mathrm{std})=U$, where $U_\mathrm{std}=(-\frac32\epsilon,\frac32\epsilon)\times i(-\frac32\epsilon,\frac32\epsilon)\subset\C$. 
 
We choose an $\epsilon>0$ such that $\Gamma$ is covered by a finite collection $\{\frac{2}{3}U_j\}$, where each $U_j$ is a neighborhood as described above.
Furthermore, we may arrange that for any two points on $\Gamma$ of Euclidean distance less than $\epsilon$ there exists an index $j$ such that they both lie in $\frac{2}{3}U_j$.

Given this, we define
$V\coloneqq \C\times\D^\circ_{\epsilon^{2n}}$, where $\D^\circ_{\epsilon^{2n}}\subset\C$
denotes the disc centered at $0$ of radius $\epsilon^{2n}$. One checks that $V_j\coloneqq\frac{2}{3}U_j\times \D^\circ_{\epsilon^{2n}}$ yields a finite set of open subsets of $V$ such that $\Ima(\Psi_n)\cap V\subset \bigcup_jV_j$.

We consider one $U_j$. Let $f$ and $L$ be the $1$--Lipschitz function and the isometry as described above. Without loss of generality, we take $L$ to be the identity.
Finally, we let $\widetilde{f}\colon (-\frac32\epsilon,\frac32\epsilon)\to S^1$ be the factorization of $f$ through $\alpha$; that is, we have $f=\alpha\circ \widetilde{f}$.

For $t\in (-\epsilon,\epsilon)$ and $r\in [0,\epsilon)$, we let $\eta_r(t)$ be the unique element in $[0,\epsilon/2)$ such that
\[\left|(t+\eta_r(t)+if(t+\eta_r(t))-(t-\eta_r(t)+if(t-\eta_r(t)))\right|=r.\]
(Existence and uniqueness of $\eta_r(t)$ follows by applying Lemma~\ref{lem:etaunique}, which also gives continuity of $(r,t) \to \eta_r(t)$.)
With this we can continuously parametrize $\Ima(\Psi_n)\cap V_j$ via:

\[P: (-\epsilon,\epsilon)\times[0,\epsilon)\to \C\times\C, (t,r)\mapsto \Psi_n\left(\left\{\widetilde{f}(t+\eta_r(t)),\widetilde{f}(t-\eta_r(t))\right\}\right).\]
In particular, $P$ factors through \[\widetilde{P}\colon(-\epsilon,\epsilon)\times[0,\epsilon)\to\M, (t,r)\mapsto \left\{\widetilde{f}(t+\eta_r(t)),\widetilde{f}(t-\eta_r(t))\right\}.\]

With this setup one now readily verifies~\eqref{item:1},~\eqref{item:2},~and \eqref{item:3}. We provide details.

\subsection*{(1)}
We claim that $\Psi_n$ is injective on $A_\epsilon\coloneqq\left\{\{s,t\}\in\M\;\bigm|\;|\alpha(s)-\alpha(t)|<\epsilon\right\}=\Psi_n^{-1}(\C\times \D^\circ_{\epsilon^{2n}})$.
Indeed, for $\{s,t\}$ and $\{s',t'\}$ in $A_\epsilon$ with $\Psi_n(\{s,t\})=\Psi_n(\{s',t'\})$, let $j$ be such that $\alpha(s)$ and $\alpha(t)$ are in $\frac{2}{3}U_j$. Note that
then also $\alpha(s')$ and $\alpha(t')$ are in $U_j$ and since they have the same midpoint as $\alpha(s)$ and $\alpha(t)$.
Since $L$ is the identity,
\begin{align*}\{s,t\}
&=
\widetilde{P}
\left(\mathrm{Re}\left(\frac{\alpha(s)+\alpha(t)}{2}\right),|\alpha(s)-\alpha(t)|\right)\\
&=
\widetilde{P}
\left(\mathrm{Re}\left(\frac{\alpha(s')+\alpha(t')}{2}\right),|\alpha(s')-\alpha(t')|\right)
=\{s',t'\}.\end{align*}

\subsection*{(2)}
For $r\in(0,\epsilon)$, we set $N\coloneqq \C\times\D_{r^{2n}}$. %Of course $N$ is a regular neighborhood of $\C\times\{0\}$, and using $P$ 
One swiftly checks that $\partial N$ intersects the locally flat-surface
$\Psi_n(\M)\cap V$ transversally. Indeed, identifying
\[
V_j\setminus(\C\times\{0\})=(-\epsilon,\epsilon)\times i(-\epsilon,\epsilon)\times (0,\epsilon)\times S^1,
\]
we see that $\partial N\cap (V_j\setminus(\C\times\{0\}))$ is given as $(-\epsilon,\epsilon)\times i(-\epsilon,\epsilon)\times \{r\}\times S^1$, while $\Psi_n(\M)\cap (V_j\setminus(\C\times\{0\}))$
is the graph of a continuous function on the first and third coordinate to the second and forth coordinate. In particular, every point $\partial N\cap\Psi_n(\M)$ has a neighborhood $W\subset \C^2$ such that the triple $(W,W\cap \partial N, W\cap\Psi_n(\M))$ is homeomorphic to $(\R^4,\R\times\R\times\{0\}\times\R,\R\times\{0\}\times\R\times\{0\})$. 

\subsection*{(3)}
Fix $r\in[0,\epsilon)$.
$\Ima(\Psi_n)\cap \C\times S^1_{r^{2n}}$ projects to the first factor $\C=\C\times\{0\}$ as a Jordan curve. For $r=0$, the projection is of course just $\Gamma$. For $r>0$,
using the parametrizations $P$ for each $\frac{2}{3}U_j$ one sees that $\Psi_n$ concatenated with the projection to the first factor maps
\[
\gamma_r\coloneqq  \left\{\{s,t\}\in\M\;\bigm|\;|\alpha(s)-\alpha(t)|=r\right\}=\Psi_1^{-1} (\C\times S^1_{r})=\Psi_n^{-1} (\C\times S^1_{r^{2n}})
\]
homeomorphically onto its image $\Gamma_r$ (by point (1) above, it is injective), and $\Gamma_r$ is homeomorphic to $S^1$ since it is a $1$--submanifold (check in a chart $\frac{2}{3}U_i$) and it is connected (since it is the image of $S^1$ via a continuous map).
Hence, $K\coloneqq\Ima(\Psi_n)\cap \C\times S_{r^{2n}}^1$ is a torus knot since it is parametrized by $\Psi_n$ restricted to the simple closed curve $\gamma_r$ and sits on the torus $\Gamma_r\times S^1_{r^{2n}}$. Given that $\gamma_r$ maps homeomorphically to the first factor of $\Gamma_r\times S^1_{r^{2n}}$, it is clear that $K$ is a $K_{1,\ell}$ torus knot.
Finally, $\Psi_n$ restricted to $\gamma_r$ and projected to the second factor has degree $(\pm2n)$; that is, $\ell = \pm 2n$ as desired. This can for example be seen as follows.

Let $\phi$ from $S^1$ to $S^1_r$ be defined by $s\mapsto \alpha(s)-\alpha(t(s))$, where $t(s)$ is the unique element with $|\alpha(s)-\alpha(t(s))|=r$ and $t(s)$ lies before
$s$ on $S^1$ (assuming $\epsilon$ is sufficiently small, we have $|s-t|<2$, whenever $|\alpha(t)-\alpha(s)|=r$; hence, \emph{lying before} is well-defined as the unique shortest path from $t(s)$ to $s$ on $S^1$ having the same orientation as the orientation induced by $S^1$). The map $\phi$ is a degree-$(\pm1)$ map from $S^1$ to $S^1_r$. To see this, we view $\phi$ as a map to $\C\setminus\{0\}$ (rather than $S^1_r)$. The map $\phi$ is homotopic to $\alpha(t)-p$, where $p\in\C$ is any point in the bounded component $B$ of the complement of $\Gamma$, and, of course, $\alpha(t)-p$ has degree $\pm1$ (that is, it induces an isomorphism on $H_1(\cdot,\Z)$) as desired). Indeed, one may continuously deform $\phi(s)$ into $\alpha(s)-p$ via $\alpha(s)-\alpha_R(t(s))$, where $[0,1]\times S^1\ni(R,t)\mapsto \alpha_R(t)$ is continuous, $p$ for $R=0$ and injective on $(0,1]\times S^1$ (in other words, it is a parametrization of the closed disc $\Gamma\cup B$ in polar coordinates, which exists by the Jordan Schoenflies theorem).
Since $\phi$ has degree $\pm1$, its $2n^{\rm th}$ power has degree $\pm2n$.
\end{proof}

%We end with a remark about the midpoint curve $\Gamma_r$.
In the above proof we in particular noted that, for locally $1$--Lipschitz $\Gamma$, there exists an $\epsilon>0$ such that, for all $r<\epsilon$, $\Gamma$ has no inscribed rectangle with diameter $r$.
It would be interesting to compare the condition \emph{$\Gamma$ has no inscribed rectangles of diameter $r$} with the condition of non-existence of  special
trapezoids considered by Matschke in~\cite[Chapter~2, Theorem 2.5]{Matschke-PhD} (see also~\cite[Theorem~4]{Matschke14}).

% !TEX root = ../squarepeg.tex

\section{M\"obius bands in $\ed$}\label{s:nomob}

The goal of this section is to prove Theorem~\ref{thm:nomob}, which states that, for $n\in\Z_{>0}$ not a square, the torus knot $K_{1,2n}\subset\es$ does not bound a locally-flat M\"obius band in $\ed$.

Unless explicitly stated, homology will be taken with integer coefficients.
%;
%we will denote by $\F = \Z/2\Z$ the field with two elements.
We will keep the notation throughout the section.

We will argue by contradiction.
To this extent, suppose that $j: M\hookrightarrow \ed$ is a locally flat M\"obius band whose boundary is $K$;
since $K$ has algebraic winding number $2n$, it is easy to see that the map $j_*: H_1(M) \to H_1(\ed)$, after choosing a generator for $H_1(M;\Z) \cong \Z$ and $H_1(\ed;\Z) \cong \Z$, is multiplication by $n$.
We call $E_K$ the exterior of $K$ in $\es$, and $E_M$ the exterior of $M$ in $\ed$.

\begin{lemma}\label{l:BDCextension}
We have $H_1(E_K) \cong \Z \oplus \Z/2n\Z$ and $H_1(E_M) \cong \Z \oplus \Z/2\Z$. In both cases, the torsion subgroup is generated by the meridian $\mu$ of $K$ and the free part can be chosen to be generated by a curve $\phi  = \{\star\}\times S^1 \subset \es \subset \ed$.
\end{lemma}

It follows that we can consider double covers of $\ed$, branched over $M$, whose boundaries are double covers of $\es$, branched over $K$. Such covers are identified with homomorphisms $\pi_1(E_M) \to \Z/2\Z$ that map $[\mu]$ to $1$; since $\Z/2\Z$ is abelian, these homomorphisms factor through $H_1(E_M) = \langle [\phi], [\mu]\rangle$.

We will refer to a specific double cover, denoted with $\Sigma(M)$, namely the one associated to the map $H_1(E_M) \to \Z/2\Z$ that sends $[\phi]$ to $0$. It is now easy to see that this induces a unique surjective homomorphism $\pi_1(E_K) \to \Z/2\Z$. We denote with $\Sigma(K)$ associated double cover, i.e.~$\partial \Sigma(M)$.
We also denote with $\tilde E_M$ the double cover of $E_M$ associated to the homomorphism above.

\begin{proof}
Let $N$ be a tubular neighbourhood of $M$ (for uniqueness and existence of such, also called normal vector bundles, see~\cite[Section~9.3]{FreedmanQuinn_90_TopOf4Manifolds}).
We know that $N$ retracts onto $M$, and that its boundary is the union of a neighbourhood of $K$ in $\es$ and the `vertical' boundary $V$, which is the non-orientable circle bundle over $M$.
In particular, $V$ retracts onto a Klein bottle, and $H_1(V) = \Z\oplus\Z/2\Z$.
More precisely, the torsion subgroup in $H_1(V)$ is generated by the fibre of the circle bundle;
since a meridian of $K$ gives a fibre, we see that the meridian of $K$ generates the torsion of $H_1(V)$.

From the Mayer-Vietoris exact sequence for $\ed = N \cup_V (E_M)$ we extract
\[
H_2(\ed) = 0 \longrightarrow H_1(V) \longrightarrow H_1(N) \oplus H_1(E_M) \longrightarrow \Z = H_1(\ed) \longrightarrow 0.
\]
Since the last group in the sequence is free, the sequence splits, and the inclusion $V\hookrightarrow E_M$ induces an isomorphism of the torsion part of $H_1(V)$ onto the torsion of $H_1(E_M)$.

Now observe that one fibre $\phi= \{\star\}\times S^1$ of $\es$ is disjoint from $K$ by construction;
it is easy to see (e.g. since the map $H_1(E_K) \to H_1(\es)$ is onto) that $[\phi]$ generates the free part of $H_1(E_K)$.
Composing with the inclusions $E_K \hookrightarrow E_M$ and $\es \hookrightarrow \ed$ also shows that $[\phi]$ generates (a choice of) the free part of $H_1(E_M)$.
\end{proof}

\begin{lemma}\label{l:SigmaK}
The $3$--manifold $\Sigma(K)$ is obtained by doing $(-n)$--surgery on the components of the $T_{2,2n}$ torus link in $S^3$;
in particular, $b_1(\Sigma(K)) = 1$.
\end{lemma}

\begin{proof}
We refer to Figure~\ref{f:Sigma245}.
The knot $K$ is presented as the closure of the $2n$--braid $\sigma_1\cdots\sigma_{2n-1}$, after doing $0$--surgery on the axis $A$.
Since the link comprising $A$ and $K$ (viewed as a link in $S^3$) is symmetric, $\Sigma(K)$ is presented as surgery on the closure of the braid $(\sigma_1\cdots\sigma_{2n-1})^2$, i.e.~the $T_{2,2n}$ torus link.
Determining the surgery coefficient is an easy calculation (see~\cite[Section 10.C]{Rolfsen}).

Now, this gives a presentation of $H_1(\Sigma(K))$ by the matrix
\[
\left(
\begin{array}{cc}
-n & n\\
n & -n
\end{array}
\right),
\]
hence $H_1(\Sigma(K)) = \Z \oplus \Z/n\Z$.
\end{proof}

Let $C$ be the 4--manifold with boundary $\Sigma(K)$ given by the surgery presentation of Lemma~\ref{l:SigmaK} above (also called the \emph{trace} of that surgery).

\begin{lemma}\label{l:Cthecap}
The $4$--manifold $C$ has homology groups $H_0(C)\cong \Z$, $H_2(C) \cong \Z^{\oplus 2}$, and all its other homology groups vanish.
The inclusion of $\Sigma(K)$ into $C$ induces an injection $H_2(\Sigma(K)) \to H_2(C)$. Moreover, $C$ contains a surface $S$ of self-intersection $-n$; that is $[S]\cdot[S] = -n$, where $\cdot$ denotes the intersection form on $H_2(C)$.
\end{lemma}

\begin{proof}
This is immediate from the fact that $C$ has a handle decomposition with no 1--, 3--, or 4--handles, and with two 2--handles.
The surface of self-intersection $-n$ is obtained by capping off a Seifert surface of either attaching circle (see Figure~\ref{f:Sigma245}(C)) with the core of the corresponding 2--handle.
\end{proof}

\begin{lemma}\label{l:H3Etilde}
The group $H_3(\tilde E_M;\F)$ is trivial.
\end{lemma}

\begin{proof}
We look at the Gysin sequence associated to the double cover $\tilde E_M \to E_M$:
\[
H_3(E_M;\F) \longrightarrow H_3(\tilde E_M;\F) \longrightarrow H_3(E_M;\F);
\]
thus, it suffices to know that $H_3(E_M;\F) = 0$.
In fact, from the Mayer--Vietoris long exact sequence of $\ed = N \cup E_M$, we extract
\[
0 = H_3(V;\F) \longrightarrow H_3(N;\F) \oplus H_3(E_M;\F) \longrightarrow H_3(\ed;\F) = 0,
\]
which implies the claim.
\end{proof}

\begin{figure}
\centering
    \begin{subfigure}[t]{0.25\textwidth}
    \centering
    \labellist
    \pinlabel $\displaystyle \pm \frac1{2n}$ at 24 60
    \pinlabel $\vdots$ at 50 62
    \pinlabel $\phantom{a}_{f}$ at 56 0 
    \endlabellist
        \includegraphics[width=0.55\textwidth]{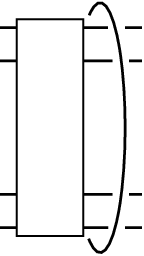}
        \caption{A surgery presentation of $T_{2n,2n\pm 1} \subset S^3$ (if $f=1)$ and of $K_{1,2n} \subset \es$ (if $f=0$).}
    \end{subfigure}
    \hspace{0.03\textwidth}
    \begin{subfigure}[t]{0.25\textwidth}
    \centering
    \labellist
    \pinlabel $\displaystyle \pm \frac1{2n}$ at 24 60
    \pinlabel $\vdots$ at 50 62
    \pinlabel $\phantom{a}_{f}$ at 0 6
    \endlabellist
        \includegraphics[width=0.55\textwidth]{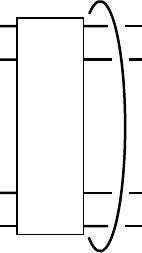}
        \caption{The same surgery presentations of $T_{2n,2n\pm1}$ and $K_{1,2n}$, after an isotopy.}
    \end{subfigure}
    \hspace{0.03\textwidth}
    \begin{subfigure}[t]{0.25\textwidth}
    \centering
    \labellist
    \pinlabel $\displaystyle \pm \frac2{2n}$ at 24 60
    \pinlabel $\vdots$ at 50 62
    \pinlabel $\phantom{a}_{f \mp n}$ at 70 5
    \pinlabel $\phantom{a}_{f \mp n}$ at 70 21 
    \endlabellist
        \includegraphics[width=0.55\textwidth]{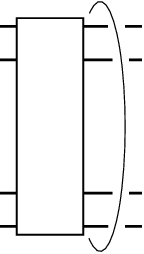}
        \caption{A surgery presentation of $\Sigma(T_{2n,2n\pm 1})$ (if $f=-1$) and $\Sigma(T_{1,2n})$ (if $f=0$). (The thin curve is auxiliary.)}
    \end{subfigure}

\caption{The $4$--manifold $C$ from Lemma~\ref{l:Cthecap} ($f=0$) and the $4$--manifold $W$ from Lemma~\ref{l:surgerypresSigmatorus} ($f=1$) as traces of surgery presentations. The picture represents closures of $2n$--braids, and each box represents a fraction of a full twist.}\label{f:Sigma245}
\end{figure}
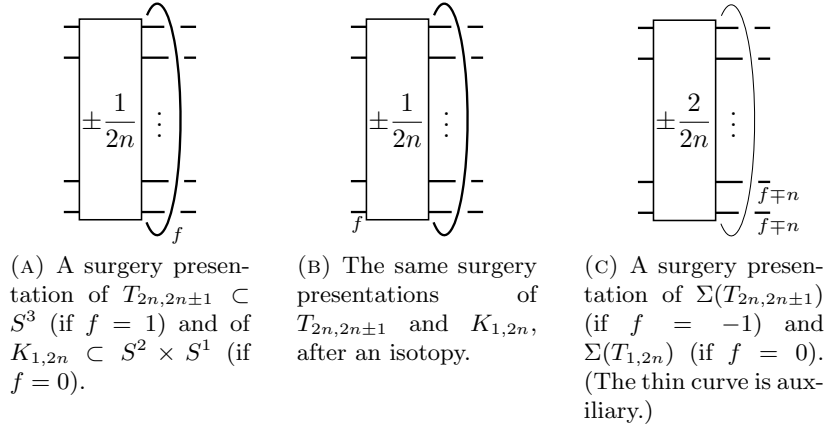

\begin{lemma}\label{l:HSigmaM}
We claim the following facts about the homology of $\Sigma(M)$.
\begin{enumerate}
\item\label{l:H3SigmaM}The third homology group $H_3(\Sigma(M))$ is torsion, and has odd order.
It follows that also $H^3(\Sigma(M))$ and $H_1(\Sigma(M),\Sigma(K))$ are torsion.

\item\label{l:H2SigmaM} The second homology group $H_2(\Sigma(M))$ is torsion, and $b_1(\Sigma(M)) = 1$.
\end{enumerate}
\end{lemma}

\begin{proof}[Proof of Lemma~\ref{l:HSigmaM}]
To prove point~\eqref{l:H3SigmaM}, we consider the Mayer--Vietoris long exact sequence associated to $\Sigma(M)  = \tilde N \cup_{\tilde V} \tilde{E}_M$, where $\tilde N$ is a neighbourhood of the branching set $\tilde M$ of $\Sigma(M) \to \ed$.
Like above, the neighborhood $\tilde N$ retracts onto $\tilde M$, and $\tilde V$ retracts onto a Klein bottle.
We then have
\[
H_3(\tilde E_M) \oplus H_3(\tilde N) \longrightarrow H_3(\Sigma(M)) \longrightarrow H_2(\tilde V) = 0;
\]
the claim follows since $H_3(\tilde E_M)$ is torsion of odd order ($H_3(\tilde E_M;\F) = 0$) and $H_3(N) = 0$.
The second part of the claim follows from the universal coefficient theorem and Poincar\'e--Lefschetz duality.

We now claim he map $H_1(\Sigma(K);\Q) \to H_1(\Sigma(M);\Q)$ induced by the inclusion is onto: this follows immediately from the long exact sequence for the pair $(\Sigma(M),\Sigma(K))$ and point~\eqref{l:H3SigmaM}:
\[
H_1(\Sigma(K);\Q) \longrightarrow H_1(\Sigma(M);\Q) \longrightarrow H_1(\Sigma(M),\Sigma(K);\Q) = 0.
\]
As a consequence, $b_1(\Sigma(M)) \le 1$.

Finally, to prove~\eqref{l:H2SigmaM}, observe that $b_k(\Sigma(M)) = 0$ for each $k\ge 3$ (the case $k=3$ is point~\eqref{l:H3SigmaM} above), $b_1(\Sigma(M)) \le 1$, and $b_0(\Sigma(M)) = 1$ ($\Sigma(M)$ is connected).
Since $\Sigma(M)$ is the double cover of $\ed$ branched over $M$, and that all three of $\chi(\ed)$, $\chi(M)$, and $\chi(V)$ vanish.
In particular, $\chi(\Sigma(M)) = 0$, too, and therefore $b_2(\Sigma(M)) = b_1(\Sigma(M)) - 1$;
however
\[
0 \le b_2(\Sigma(M)) = b_1(\Sigma(M)) - 1 \le 1-1 = 0;
\]
Therefore $b_1(\Sigma(M)) = 1$ and $b_2(\Sigma(M)) = 0$.
\end{proof}

Let us call $X$ the 4--manifold obtained by gluing $\Sigma(M)$ and $-C$ (the manifold given by $C$ with reversed orientation) along $\Sigma(K)$.

\begin{lemma}
The second homology group $H_2(X;\Q)$ is $1$--dimensional.
\end{lemma}

\begin{proof}
We look at the Mayer--Vietoris long exact sequence; keeping in mind that $H_*(C) = H_*(-C)$, that $H_1(C) = 0$, and that $H_2(\Sigma(M);\Q) = 0$ (Lemma~\ref{l:H2SigmaM}), we obtain:
\[
H_2(\Sigma(K);\Q) \longrightarrow H_2(C;\Q) \longrightarrow H_2(X;\Q) \longrightarrow H_1(\Sigma(K);\Q) \longrightarrow H_1(\Sigma(M);\Q).
\]
The first map and the last map are injections by Lemma~\ref{l:Cthecap} and Lemma~\ref{l:H2SigmaM}, respectively. We therefore have a short exact sequence
\[
H_2(\Sigma(K);\Q) \longrightarrow H_2(C;\Q) \longrightarrow H_2(X;\Q) %\longrightarrow 0
,
\]
where the second vector space has dimension 2 by Lemma~\ref{l:Cthecap}, hence $b_2(X) = 1$.
\end{proof}

\begin{proof}[Proof of Theorem~\ref{thm:nomob}]
Since $-C \subset X$, $X$ contains a surface of self-intersection $+n$, by Lemma~\ref{l:Cthecap}.
However, $b_2(X) = 1$, so the intersection form is unimodular of rank $1$, and it contains a vector of positive square, so it is $\langle +1 \rangle$.
This contradicts the existence of a vector of square $+n$, since $n$ is not a square.
\end{proof}
We note that Proposition~\ref{prop:T2p2kppm1doesnotboundMB}, given below, provides another proof of Theorem~\ref{thm:nomob}, for $n$ a prime, while Proposition~\ref{p:T45} implies Theorem~\ref{thm:nomob} for $n\equiv 1, 2 \pmod 4$.
% !TEX root = ../squarepeg.tex

\newcommand{\gtop}{\gamma_4^{\rm top}}
\newcommand{\gsmooth}{\gamma_4}
\newcommand{\gstop}{\gamma_4^{\rm stop}}
\newcommand{\gssmooth}{\gamma_4^{\rm s}}

\section{Non-orientable $4$--genus for torus knots in $S^3$}\label{s:balls}

In this section we discuss non-orientable $4$--genus for torus knots in $S^3$. In Subsection~\ref{subsec:ProofThm1.3}, we show that for torus knots the notion of non-orientable $4$--genus depends on the choice of category by establishing the existence of locally-flat surfaces in $B^4$ with boundary certain torus knots. In contrast, in Subsection~\ref{subsec:mob}, we discuss obstructions (in both categories) for the existence of Möbius bands in $B^4$ with boundary a given knot.

\subsection{Proof of Theorem~\ref{thm:gammafortorusknots}}\label{subsec:ProofThm1.3}
The following proposition is a strengthening of Theorem~\ref{thm:gammafortorusknots}.

\begin{prop}\label{prop:gammatop<gammasmooth}
For integers $n\geq5$, we have $\gtop(T_{2n-1,2n}) \le n-2 < n-1 = \gsmooth(T_{2n-1,2n})$.
In fact, there exists a non-orientable connected locally-flat surface $\Sigma\subset B^4$ with $b_1(\Sigma)=n-2$ such that $\partial \Sigma=T_{2n-1,2n}$ and $\pi_1(B^4\setminus \Sigma)\cong \Z/2\Z$.
\end{prop}

\begin{remark}
Recently, Lobb observed in~\cite{Lobb19} that $\gsmooth$ is smaller on torus knots than previously conjectured by Batson in~\cite{Batson}.
However, one may ask whether the conjecture holds when restricting to surfaces with complements that have cyclic fundamental group.
(Note that, by the Mayer--Vietoris sequence, the first homology of the complement of such a surface is always $\Z/2\Z$.)
The second assertion of the above proposition, implies that even in this more restrictive setup, the topologically locally-flat quantity is strictly smaller than the corresponding smooth one.
Note that Proposition~\ref{p:T45} at least says that if $2n-1 \equiv 5 \pmod 8$, then we cannot hope to decrease the locally-flat non-orientable genus all the way down to~$1$.
\end{remark}

\begin{proof}[Proof of Proposition~\ref{prop:gammatop<gammasmooth}]
The equality $n-1 = \gsmooth(T_{2n-1,2n})$ is due to Batson.
For the upper bound on the topological cross cap number, we find a non-orientable spanning surface $S \subset S^3$ for $T_{2n-1,2n}$ with $b_1(S) = n$ with the following property: there is a separating simple closed curve $\gamma \subset S \subset S^3$ with trivial Alexander polynomial, such that one of the two connected components of $S\setminus\gamma$ is a once-punctured torus.

We modify $S$ by replacing the once-punctured torus in $S^3$ with the locally-flat disc in $B^4$ with boundary $\gamma$. By doing so, we find a locally-flat non-oriented surface $\Sigma \subset B^4$ with $b_1(\Sigma) =n-2$ that fills $T_{2n-1,2n}$.
The existence of such a disc is guaranteed by a consequence Freedman's celebrated disk theorem: knots with Alexander polynomial one are topologically slice.
In fact, a knot $K$ has Alexander polynomial $1$ if and only if, there exists a locally-flat disc $D\subset B^4$ with boundary $K$ such that $\pi_1(B^4\setminus D)\cong\Z$; see~\cite[Theorem~1.13]{Freedman_82_TheTopOfFour-dimensionalManifolds}. A Seifert--van Kampen calculation allows to check that the complement of $\Sigma$ (pushed into $B^4$ to be properly embedded) has cyclic fundamental group; this follows from the fact that $\pi_1(B^4\setminus D)\cong\Z$. In the rest of the proof we implement this in detail.

We first describe a (non-orientable) spanning surface $S$ for $T_{2nk-1,2n}$ for $k\geq0$. While we are interested in the case $k=1$, it is instructive to see that all of these are built from the case $k=0$. For this, we view $T_{2nk-1,2n}$ as the closure of the $2n$--stranded braid $(\sigma_1\sigma_2\cdots \sigma_{2n-1})^{2kn-1}$.
For $k=0$, we take $S$ to be the checkerboard surface for the standard diagram of the braid closure of $(\sigma_1\sigma_2\cdots \sigma_{2n-1})^{-1}$ that misses the braid axis. For $k\geq 1$, we take $S$ to be the surface obtained from the one for $k=0$ by $k$--surgery along the braid axis; i.e.~by adding $k$ positive full twists as depicted in Figure~\ref{fig:S}. Note that $b_1(S)=n$.
\begin{figure}[ht]
\centering
\subcaptionbox{\label{fig:Fzero} The spanning surface $S$ (gray).}
{\def\svgscale{0.6}
\centering
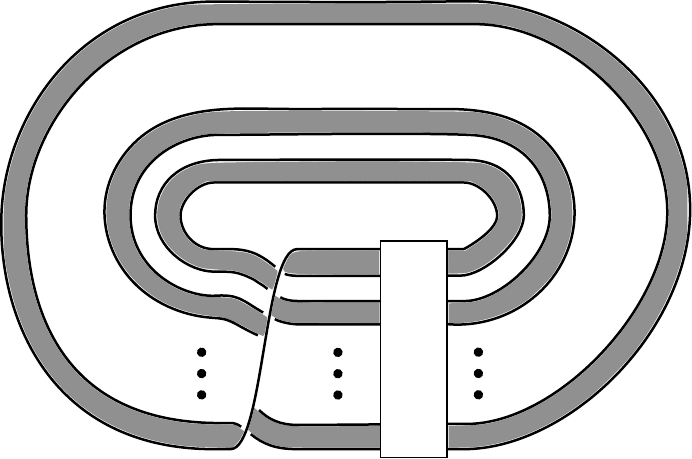}
\subcaptionbox{\label{fig:alphasinFzero} The curves $\alpha_i$: $\alpha_{i+1}$ is a copy of $\alpha_i$ shifted down by two strands. }
{\def\svgscale{1.3}
\centering
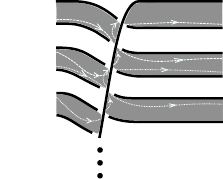}
\caption{\label{fig:S}The surface $S$ and the curves $\alpha_i$ on $S$. The curves $\alpha_i$ and $\alpha_{i+1}$ intersect in one point. The box represents $k$ positive full twists on $2p$ strands.}
\end{figure}

Next, we define an orientable, incompressible subsurface $F$ of $S$ with $b_1(F)=n-1$.
Let $\alpha_1$, $\alpha_2$, $\dots$, $\alpha_{n-1}$ be the simple closed curves depicted in Figure~\ref{fig:S}, and take $F$ to be a neighborhood of their union; in particular, the homology classes of the $\alpha_i$ constitute a basis for $H_1(F;\Z)$.
We focus on $F$ and try to find a separating curve in $F$ with Alexander polynomial one that cuts out a once-punctured torus. The boundary of a once-punctured torus $T$ is a knot with Alexander polynomial $1$ if and only if one (and thus all) matrix $A$ representing the Seifert form on $T$ satisfies
\begin{equation}\label{eq:Alex1}
\det(t^{1/2}A-t^{-1/2}A^T)=1.
\end{equation}
Therefore, finding such a once punctured torus amounts to finding a pair of once-intersecting curves $\beta_1$ and $\beta_2$ in $F$ such that, if $A$ is the matrix representing the restriction of the Seifert form on $F$ to ${\rm span}\{[\beta_1],[\beta_2]\}\subset H_1(F,\Z)$, then $A$ satisfies~\eqref{eq:Alex1}.
In fact, it suffices to find a rank-two subgroup $H$ of $H_1(F;\Z)$ on which the Seifert form is given by a bilinear form that satisfies~\eqref{eq:Alex1} since any pair of such homology classes can be represented by once intersecting simple closed curve; see~\cite[Proposition~9]{FellerLewark_16} or~\cite[last paragraph of the proof of Proposition~6]{FellerMcCoy_15} for details.
Consequently, for the rest of the proof, we only concern ourselves with the Seifert form on $F$ and subgroups $H$ as described above.

The Seifert form of $H_1(F;\Z)$ with respect to the basis $([\alpha_1],\dots,[\alpha_{n-1}])$ is
\[
M_k=M_0+M_{\rm twist}=
\left(\begin{array}{ccccc}-1&-1&0&0&\\
0&-1&-1&0&\\
0&0&-1&-1&\\
&&& & \ddots\end{array}\right)
+k\left(\begin{array}{cccc}
4&4&4&\\
4&4&4&\cdots\\
4&4&4&\\
 & \vdots && \ddots
\end{array}\right);
\]
in particular, for $n=5$
\begin{equation}\label{eq:Mforp=5}M_k=\left(\begin{array}{cccc} 4k-1 & 4k-1 & 4k & 4k\\
4k & 4k-1 & 4k-1&4k\\
4k & 4k & 4k-1 & 4k-1\\
4k & 4k & 4k & 4k-1\end{array}\right).\end{equation}
We note that $M_0+M_0^T$ is negative definite. In fact, for $k=0$, $F$ is the minimal Seifert surface of the $T_{2,-n+1}$ torus link.

We have now reduced the problem to the following linear algebra question:
for $n\geq5$ and $k=1$, do there exist vectors $a$ and $b$ in $\Z^{n-1}$ such that
$a^TM_ka=0$, $a^TM_kb=0$, and $b^TM_ka=\pm1$? And, once we have such an $H$ for $n=5$, we have it for all $n\geq 5$ since $M_k$ for $n=5$ is the top-left $4\times4$ sub-matrix of $M_k$ for $n\geq 5$.
Hence, we conclude the proof by noting that the following pair of vectors has the desired property: $\left(a=( -1,  -1,   0,   1)^T, b=(4,   1,   2,  -4)^T\right)$.
\end{proof}

\subsection{Torus knots that do not bound M\"obius bands}\label{subsec:mob}

Recall that Batson~\cite{Batson} proved that $\gsmooth(T_{2n-1,2n}) = n-1$ for each $n \ge 2$; however, neither his result nor subsequent related developments~\cite{OSSz, GM} say anything about $\gsmooth(T_{2n,2n+1})$.
For this reason, for most of this section we focus on finding values of $n$ for which we can prove that $\gtop(T_{2n,2n+1}) > 1$.

We begin with a statement about double covers of torus knots $T_{2n, 2n\pm1}$.
Recall that, if $p$ is even and $q$ is odd and coprime with $p$, then $|{\det T_{p,q}}|= |q|$.
Recall also that the double cover of $S^3$ branched over $T_{p,q}$ is the manifold $\Sigma(2,p,q)$, i.e. the link of the singularity of $\{x^2+y^p+z^q = 0\} \subset \C^3$ at the origin.

\begin{lemma}\label{l:surgerypresSigmatorus}
Let $n\geq 1$ be an integer. The $3$--manifold $\Sigma(2,2n,2n\pm1)$ is the boundary of a $4$--manifold $W$ with $H_1(W) = H_3(W) = 0$ and whose intersection form is represented by the matrix $\big({\tiny\begin{array}{cc}\mp n - 1& n\\n& \mp n - 1\end{array}}\big)$.
In particular, $H_1(\Sigma(2,2n,2n\pm1))$ is cyclic of order $2n\pm 1$.
\end{lemma}

\begin{proof} 
The $4$--manifold $W$ is given as the trace of the surgery in Figure~\ref{f:Sigma245}. We first unknot $T_{2n,2n\pm 1}$ with an axial surgery (left); the corresponding link is symmetric, so we swap its components (center); finally, we branch double cover over the axis (right).
The framings are easily computed (see, for example,~\cite[Section 10.C]{Rolfsen}) to be $\mp n - 1$ for each of the two $2$--handles, and their linking number is $\pm n$ (the signs are coherent with the sign determining the knot).
Up to changing the orientation of one of the two components, we can always change the signs off the diagonal, so the intersection form is presented by the matrix $V = \big({\tiny\begin{array}{cc}\mp n - 1& n\\n& \mp n - 1\end{array}}\big)$.

Adding the second row of $V$ $\pm1$-times to the first and then adding the first row $n$-times to the second yields $V\sim \big({\tiny\begin{array}{cc} - 1& \mp1\\ 0& \mp 2n - 1\end{array}}\big)$. Hence $H_1(\Sigma(2,2n,2n\pm1))\cong \Z/(2n\pm1)\Z$ as claimed.
\end{proof}

In what follows, given a rational number $q$, we denote with $\langle q \rangle$ the bilinear form $x\otimes y \mapsto xqy$; we use the same notation for bilinear forms that take values in $\Z$ (like the intersection form of a $4$--manifold with $b_2 = 1$) and for those that take values in $\Q/\Z$ (like the linking form $\lambda_Y$ of a rational homology $3$--sphere $Y$ with cyclic $H_1$).
We denote with $\big(\frac pq\big)$ the Jacobi symbol.
We say that a linking form $\lambda_Y$ \emph{represents $s$ as a square} if there exists a torsion element $x \in H_1(Y)$ such that $\lambda_Y(x,x) = s$.
For any integer $p>0$ dividing the order $|G|$ of a finite Abelian group $G$ with a bilinear form $\lambda$ to $\Q/\Z$, we take its reduction modulo $p$ to be 
the $\Q/\Z$--valued bilinear form $\lambda_p$ on $G\otimes \Z/p\Z$ given by
$(x\otimes \overline{1})\otimes(y\otimes \overline{1})\mapsto \lambda(x,y)\tfrac{|G|}{p}$. Note that, by definition, if $\lambda$ represents $\tfrac{k}{|G|}$ as a square, then $\lambda_p$ represents  $\tfrac{k}{p}$ as a square. %Indeed, if $\lambda(x,x)=\tfrac{k}{|G|}$, then $\lambda_p(x\otimes \overline{1},x\otimes \overline{1})=\tfrac{k}{|G|}$

The following statement is a special case of a result of Murakami and Yasuhara~\cite{MurakamiYasuhara}.

\begin{prop}[{\cite[Corollary~2.7]{MurakamiYasuhara}}]\label{p:MY}
Let $K\subset S^3$ be a knot that bounds a locally-flat M\"obius band in $B^4$.
Then there is a $\lambda_{\Sigma(K)}$--orthogonal decomposition $H_1(\Sigma(K)) = G \oplus H$ where $H$ has square order and $\lambda_{\Sigma(K)}|_G$ represents $\frac{1}{|G|}$ or $-\frac{1}{|G|}$ as a square.

In particular, if $\det K$ is square-free, then the linking form on $\Sigma(K)$ represents $\frac{1}{|{\det K}|}$ or $-\frac{1}{|{\det K}|}$ as a square.
\end{prop}

We will need the following elementary number theory calculation.

\begin{lemma}\label{l:presentingassquares}
Let $p$ be a prime and $\lambda$ be a linking form on $\Z/p\Z$ that represents $-\frac1p$ as a square.
\begin{enumerate}
\item\label{linking1mod4} If $p\equiv 3 \pmod 8$, then $\lambda$ represents $\frac2p$ as a square, but not $-\frac2p$.
\item\label{linking2mod4} If $p\equiv 5 \pmod 8$, then $\lambda$ represents neither $\frac2p $ nor $-\frac2p$ as squares.
\item\label{linking3mod4} If $p \equiv 7 \pmod 8$, then $\lambda$ represents $-\frac2p$ as a square, but not $\frac2p$.
\end{enumerate}
\end{lemma}

\begin{proof}
Since $\lambda$ represents $-\frac1p$, $\lambda$ is isomorphic to $\langle -\frac1p \rangle$; that is, $\lambda(x,x) = -\frac{x^2}p$ for every $x \in \Z/p\Z$, and $\frac{2}{p}$ (respectively, $-\frac{2}p$) is represented as a square if and only if $\jacobi{-2}{p} = 1$ (resp. $\jacobi{2}{p} = 1$).

It is well-known that $\jacobi{-1}{p} = 1$ if and only if $p\equiv 1 \pmod 4$, and that $\jacobi{2}{p} = 1$ if and only if $p\equiv \pm 1 \pmod 8$ (see, for instance,~\cite[Theorems~82 and~95]{HardyWright}). Using multiplicativity of Legendre symbols, we quickly derive all three statements.
\end{proof}

We can now prove Proposition~\ref{p:T45} from the introduction; that is, we show that $T_{p,p\pm 1}$ does not bound a M\"obius band if $p \equiv 5 \pmod 8$.

\begin{proof}[Proof of Proposition~\ref{p:T45}]
Call $\Sigma = \Sigma(2,p,p\pm1)$ and $\lambda = \lambda_\Sigma$ its linking form.
By Lemma~\ref{l:surgerypresSigmatorus}, $H_1(\Sigma)$ is cyclic of order $p$.
Suppose towards a contradiction that $T_{p,p\pm1}$ is the boundary of a locally flat M\"obius band.
By Proposition~\ref{p:MY}, there is a $\lambda$--orthogonal decomposition $H_1(\Sigma) = G\oplus H$, where $H$ has square order and $\lambda|_G$ represents $\frac1{|G|}$ or $-\frac1{|G|}$. Since $p\equiv 5 \pmod 8$ and all odd squares are congruent to $1$ modulo $8$, $G$ is not the trivial group. In fact,
% and more precisely
$|G| \equiv 5 \pmod 8$.
This implies that either:
\begin{itemize}
\item[(i)] $|G|$ is divisible by a prime $q \equiv 5 \pmod 8$, or
\item[(ii)] $|G|$ is divisible by two primes $q_1 \equiv 3 \pmod 8$ and $q_2 \equiv 7 \pmod 8$.
\end{itemize}
We claim that, $\lambda|_G$ represent neither $\frac 2{|G|}$ nor $-\frac 2{|G|}$ as a square. We treat cases (i) and (ii) separately using reduction modulo $q$ and $q_1$ and $q_2$, respectively.

In case (i), we reduce $\lambda$ modulo $q$. Since $\lambda|_G$ represents $\frac1{|G|}$ or $-\frac1{|G|}$ as a square, its reduction modulo $q$ is a quadratic form on $\Z/q\Z$ that represents $\frac1q$ or $-\frac1q$ (and hence both, since $-1$ is a square mod $q$) as a square, therefore by Lemma~\ref{l:presentingassquares}\eqref{linking2mod4}, it represent neither $\frac2q$ nor $-\frac2q$ as a square. %Lifting back to $G$,
Hence, $\lambda|_G$ represent neither $\frac 2{|G|}$ nor $-\frac 2{|G|}$ as a square.

In case (ii), we reduce $\lambda$ modulo $q_1$ and modulo $q_2$. If $\lambda|_G$ represents $-\frac1{|G|}$ as a square, its reduction modulo $q_1$ cannot represent $-\frac2{q_1}$ as a square by Lemma~\ref{l:presentingassquares}\eqref{linking1mod4} and its reduction modulo $q_2$ cannot represent $\frac2{q_2}$ as a square by Lemma~\ref{l:presentingassquares}\eqref{linking3mod4}. Similarly, if $\lambda|_G$ represents $\frac1{|G|}$ as a square, its reduction modulo $q_2$ cannot represent $-\frac2{q_2}$ as a square and its reduction modulo $q_1$ cannot represent $\frac2{q_1}$ as a square. %Lifting both cases back to $G$
Hence, $\lambda|_G$ represents neither $\frac 2{|G|}$ nor $-\frac 2{|G|}$ as a square.

Next we derive  from Lemma~\ref{l:surgerypresSigmatorus} that $\lambda_\Sigma$ does represent $\frac2{|G|}$ or $-\frac2{|G|}$ as a square, which leads to the desired contradiction.

%\MG{Changed sign from $T_{p,p-1}$ to $T_{p,p+1}$! Same in the next paragraph}
We begin with the case of $T_{p,p+1}$. Call $n = \frac{p+1}2$.
In this case, $\Sigma(2,2n-1,2n)$ bounds a $4$--manifold $W$ with $H_1(W)=0$ and intersection form presented by $\big({\tiny\begin{array}{cc}n-1 & n \\ n & n-1\end{array}}\big)$, and therefore the linking form of $\Sigma(2,2n-1,2n)$ is presented by $\frac1p \big({\tiny\begin{array}{cc}n-1 & n \\ n & n-1\end{array}}\big)$. In particular, it represents $\frac{n-1}{p}$ as a square.
Since $2(n-1) = 2n-2 \equiv -1 \pmod p$, the linking form represents $\frac{n-1}{p}$ if and only if it represents $-\frac{2}{p}$ as a square. By restricting to $G$, we see that $\lambda|_G$ represents $-\frac2{|G|}$ as a square, which is a contradiction.

Let us now look at the case of $T_{p,p-1}$. Call $n = \frac{p-1}2$.
$\Sigma(T_{2n,2n+1})$ now bounds a $4$--manifold with intersection form presented by $\big({\tiny\begin{array}{cc}-n-1 & n \\ n & -n-1\end{array}}\big)$, and therefore the linking form represents $\frac{n+1}{p}$ as a square. The inverse of $n+1$ modulo $p$ is $2$, so it also represents $\frac{2}{p}$ as a square. As above, restricting to $G$ we reach a contradiction.
\end{proof}

The next proposition implies Theorem~\ref{thm:nomob} when $n$ is a prime.

\begin{prop}\label{prop:T2p2kppm1doesnotboundMB}
For each odd prime $p$ there are infinitely many positive integers $k$ such that the knot $T_{2p,2kp \pm 1}$ does not bound a locally-flat M\"obius band in $B^4$.
\end{prop}

The statement means that for each choice of a sign there are infinitely many values of $k$ such that the statement is true; i.e. we do not claim that there are values of $k$ such that the statements holds with both signs. We prefer to leave the ambiguity in order to keep the notation lighter.

\begin{proof}
Let $q$ be a prime number with $q\equiv 1 \pmod 4$ and such that $p$ is not a square residue modulo $q$. We claim that we can always construct such a number.
Indeed, by quadratic reciprocity, $\jacobi{p}{q} = \jacobi{q}{p}$; so it suffices to choose $q$ such that $\jacobi{q}{p} = -1$. There are infinitely many such primes: it suffices to pick an integer $r$ such that $\jacobi{r}{p} = -1$ and look at the arithmetic progression $r+sp$, which contains infinitely many primes by Dirichlet's theorem.
Since $q\equiv 1 \pmod 4$, $\jacobi{-1}q = 1$, so that $\jacobi{-p}{q} = -1$, too.
In particular, for any positive integer $r$, neither $-p$ nor $p$ is a square modulo $qr$.

Let $h_0$ be any positive integer such that $2h_0p \pm 1 \equiv q \pmod{q^2}$. Let us look at the arithmetic progression $2h_0p \pm 1 + hpq^2$, with $h > 0$. Since $\gcd(2h_0p\pm1, pq^2) = q$, we can re-write
\[
2h_0p \pm 1 + hpq^2 = q(a + hpq),
\]
where $\gcd(a,pq) = 1$; in particular, by Dirichlet's theorem there are infinitely many primes in the arithmetic progression $a+hpq$. Choose $h$ such that $r = a+hpq$ is a prime, and let $k = h_0 + q^2h$.
In particular, we have $2kp\pm 1 = qr$.

We claim that $T_{2p,2kp\pm 1}$ does not bound a locally-flat M\"obius band, and more precisely that $T_{2p,2kp\pm 1}$ violates the Murakami--Yasuhara criterion.

Since $qr = 2kp\pm 1$, $qr = \det T_{2p,2kp\pm 1}$; since $q$ and $r$ are distinct primes, $H_1(\Sigma(2,2p,qr)) \cong \Z/qr\Z$.
By tweaking the proof of Lemma~\ref{l:surgerypresSigmatorus}, we see that $\Sigma(2,2p,qr)$ has a surgery presentation for which the linking matrix is given by the following tridiagonal $2k\times 2k$ matrix:
\[
Q = \left({
\tiny{
\begin{array}{cccccccccccccccccccc}
-2 & 1  \\
1 & -2 \\
& & \ddots  \\
& & &  -2 & 1 & 0 & 0\\
& & & 1 & \pm p - 1 & \pm p & 0\\  
& & & 0 & \pm p & \pm p - 1 & 1\\
& & & 0 & 0 & 1 &- 2\\
& & & & & & & \ddots \\
& & & & & & & & -2 & 1 \\
& & & & & & & & 1 & -2 \\
\end{array}}
}
\right).
\]
Since $H_1(\Sigma(2,2p,qr))$ is cyclic, it suffices to compute one non-trivial square in the linking form $\lambda$; by an explicit inductive computation, the first entry of the matrix $-Q^{-1}$, which is the matrix that represents the linking form, is $\frac{qr\mp p}{qr}$.

Since $\jacobi{p}{q} =\jacobi{-p}{q}= -1$ by choice, and since $\lambda$ represents $\mp\frac{p}{qr}$ as a square by the previous computation, $\lambda$ does not represent $\frac{1}{qr}$ nor $-\frac{1}{qr}$ as a square. Thus, by Proposition~\ref{p:MY}, $T_{2p,2kp\pm1}$ is not the boundary of a locally-flat M\"obius band.
\end{proof}

We want to highlight the limitations of the techniques we used above.
%First of all, when $\det K$ is not square-free, Proposition~\ref{p:MY} above gives weaker restrictions. (The slogan is ``the proposition is true up to metabolisers''.) For instance, there is no obstruction from Proposition~\ref{p:MY} when $H_1(\Sigma(K))$ is cyclic of square order.
%
%Secondly, when $\det K$ is not a prime, it becomes (slightly) harder to determine whether $\pm \frac1{\det K}$ is represented as a square.
%
%Thirdly,
There are shortcomings to applying Proposition~\ref{p:MY} even when $|{\det K}| =: q$ is a prime: if $-1$ is not a square residue modulo $q$ (equivalently, if $q \equiv 3 \pmod 4$), then for algebraic reasons either $\frac1q$ or $-\frac1q$ is always represented by a square. So, for instance, we cannot directly conclude anything about the existence of a M\"obius band in $B^4$ whose boundary is $T_{6,7}$. However, we can find an alternative to Proposition~\ref{p:T45}.

\begin{prop}\label{p:T67}
%\MG{Upgraded from $2n+1$ prime to $2n+1$ square-free and from $T_{22,23}$ to $T_{14,15}$ (which I think fits the assumptions).}
Let $n \equiv 3 \pmod 4$ be a positive integer such that $2n+1$ is square-free and ${n+1}$ is not a square.
Then $T_{2n,2n+1}$ does not bound a \emph{smooth} M\"obius band in $B^4$.
In particular, $T_{14,15}$ and $T_{22,23}$ do not bound smooth M\"obius bands in $B^4$.
%\MG{Changed from ``a Mobius band'' to ``Mobius bands''. (Not sure that either is grammatically correct, but this one sounds slightly better to me.)}
\end{prop}

%\MG{Clarified how Dirichlet is used.}
Note that there are infinitely many integers satisfying the three conditions in the statement: for instance, $n+1$ is never a square if $n \equiv 7 \pmod {16}$, and by Dirichlet's theorem on primes in arithmetic progressions there are infinitely many primes (and in particular square-free integers) congruent to $15$ modulo $32$.

\begin{proof}
We know from Lemma~\ref{l:surgerypresSigmatorus} that $Y \coloneqq \Sigma(2,2n,2n+1) = \Sigma(T_{2n,2n+1})$ bounds a spin negative definite $4$--manifold $W$ with second Betti number $2$.
By successively blowing up $W$, we obtain the canonical negative plumbing $P$ whose boundary is $Y$, according to Neumann~\cite{Neumann}; call $\Gamma$ the associated weighted graph (which is a three-legged star-shaped graph, in this case). Ozsv\'ath and Szab\'o computed correction terms of $Y$ starting from the plumbing graph $\Gamma$~\cite{OSz-plumbing} as follows.

Let $L$ be the intersection lattice of $P$, and fix a coset in ${\rm Char}(L)/2L$ (recall that the set ${\rm Char}(L)$ of characteristic covectors of $L$ is a $2L$--torsor). Since cosets in ${\rm Char}(L)/2L$ are in bijection with spin$^c$ structures on $Y$, we will denote the coset by the spin$^c$ structure $\ft$ it corresponds to. Then:
\begin{equation}\label{e:OSz-plumbing}
%d(Y,\ft) = \max_{\xi \in \ft} \frac{\xi^2 + \rank L}4.
d(Y,\ft) = d(L,\ft) \coloneqq \max_{\xi \in \ft} \frac{\xi^2 + \rank L}4.
\end{equation}

We make an easy observation here: if $L = \langle -1 \rangle \oplus L'$, then the inclusion $L'\to L$ induces a bijection $j: {\rm Char}(L)/2L \to {\rm Char}(L')/2L'$, and it is easy to verify that
$d(L,\ft) = d(L',j(\ft))$.
%\[
%\max_{\xi \in \ft} \frac{\xi^2 + \rank L}4 = \max_{\xi \in j(\ft)} \frac{\xi^2 + \rank L'}4.
%\]
This implies that, since $P$ is a blow-up of $W$ and the intersection from changes by taking the direct sum with $\langle -1 \rangle$ under blow-ups, we can do the maximisation on the intersection lattice of $W$ (which has rank $2$), rather than on the intersection lattice of $P$.

To unravel~\eqref{e:OSz-plumbing}, let $Q$ be the $2\times 2$ matrix $\big({\tiny\begin{array}{cc}- n - 1& n\\n& -n - 1\end{array}}\big)$. Since the entries on the diagonal are both even, characteristic covectors are of the form $2Q^{-1}\eta^T$, where $\eta = (\eta_1,\eta_2)$ is an integer vector. The maximisation then reads:
\begin{equation}\label{e:explicitd}
d(Y,\ft) = \max_{\xi \in \ft} \frac{\xi^2 + \rank L}4 = \max_{2Q^{-1}\eta^T \in \ft} {\eta}\,Q^{-1}\eta^T + \frac12 = \frac12 - (\eta_1 + \eta_2)^2 - \frac{\eta_1^2 + \eta_2^2}{2n+1}.
\end{equation}
Note that the quantity on the right-hand side is \emph{negative} as soon as $\eta_1+\eta_2 \neq 0$.

Suppose that $T_{2n,2n+1}$ bounds a (locally-flat) M\"obius band $M$ in $B^4$. Then $Y$ bounds a $4$--manifold $Z$ with $b_2(Z) = 1$, namely  the double cover of $B^4$ branched over $M$. If $M$ is smoothly embedded, then $Z$ is a \emph{smooth} $4$--manifold.

Since $2n+1$ is a square-free, the intersection form of $Z$ is $\langle\pm(2n+1)\rangle$.
Suppose that $Z$ were positive definite. Then we could glue $Z$ and $-W$ along their boundary to get a closed, positive definite $4$--manifold $X$ with $b_2(X) = 3$. Its intersection form, since it is unimodular and negative definite, has to be diagonal. But this implies that the self-intersection of the generator of $H_2(Z)$, which is $2n+1$, is a sum of three squares. But this contradicts the fact that $2n+1 \equiv 7 \pmod 8$.

So $Z$ is negative definite, with intersection form $\langle -2n-1\rangle$.
However, Ozsv\'ath and Szab\'o~\cite[Theorem~9.6]{OSz-absolutely} proved, under these assumptions, for each spin$^c$ structure $\fs$ on $Z$ which restricts to $\ft$ on $Y$, that
\begin{equation}\label{e:OSz-inequality}
\frac{c_1(\fs)^2 + 1}4 \le d(Y,\ft),
\end{equation}
and that the two sides are congruent modulo $2$.

Let us focus on the case where $\langle c_1(\fs), H_2(Z) \rangle = \Z$ (this is possible since the intersection form of $Z$ is odd, and therefore there is such a spin$^c$ structure; in fact there are exactly two, which are conjugate). Then $c_1(\fs)^2 = -\frac{1}{2n+1}$, and therefore
%\[
%\frac{c_1(\fs)^2 + 1}{4}  = \frac{n}{4n+2} > 0.
%\]
$\frac{c_1(\fs)^2 + 1}{4}  = \frac{n}{4n+2} > 0$.

It follows from~\eqref{e:OSz-inequality} that $d(Y,\ft) > 0$, and therefore, from~\eqref{e:explicitd}, that $\ft$ corresponds to $Q^{-1}\eta^T$ with $\eta_1 + \eta_2 = 0$. In particular,
%\[
%d(Y, \ft) = \frac12 - \frac{2x^2}{2n+1}
%\]
$d(Y, \ft) = \frac12 - \frac{2x^2}{2n+1}$
for the integer $x=\eta_1$.

We now use the congruence condition in~\eqref{e:OSz-inequality}. We have (the reduction modulo $\frac12$ of) the congruence condition (modulo $2$) telling us that
$-1 \equiv -8x^2 \pmod{2n+1}$.
However, if $\frac{n+1}4$ is not a square, then the smallest positive integer solution $x$ of this congruence has $x > \sqrt{\frac{n+1}4}$, and therefore
\[
d(Y,\ft) = \frac12 - \frac{2x^2}{2n+1} < \frac12 - \frac{n+1}{4n+2} = \frac{n}{4n+2} = \frac{c_1(\fs)^2 + 1}{4},
\]
which contradicts~\eqref{e:OSz-inequality}.
\end{proof}

Recall that $2$--bridge links are links $L \subset \R^3 \subset S^3$ such that the restriction of the $z$--function is Morse with two minima and two maxima. The double cover of $S^3$ branched over a two $2$--bridge link is a lens space, and two $2$--bridge links are isotopic if and only if their branched covers are homeomorphic. We refer to $K_{p/q}$ as the unique $2$--bridge link whose double cover is the lens space $L(p,q)$; note that $K_{p/q}$ is a knot if and only if $p$ is odd, and that $p = |{\det K_{p/q}}|$.

\begin{prop}\label{p:2bridge}
Let $p$ a positive integer.
If $p\equiv 5\pmod 8$, the $2$--bridge knot $K_{p/(p-2)}$ does not bound a \emph{locally-flat} M\"obius band in $B^4$. If $p\equiv 7\pmod 8$ and $p>7$, the $2$--bridge knot $K_{p/(p-2)}$ does not bound a \emph{smooth} M\"obius band in $B^4$.
%The knot $K_{15/13}$ does not bound a \emph{locally-flat} M\"obius band.
\end{prop}

\begin{proof}
Suppose that $K_{p/(p-2)}$ bounds a locally-flat M\"obius band $M$ in $B^4$; then the double cover of $B^4$ branched over $M$ is a $4$--manifold $Z$ with $b_2(Z) = 1$ whose boundary is a lens space, namely $\Sigma(K_{p/(p-2)}) = L(p,p-2)$.
%\MG{I didn't like ``minus the intersection form'', so I removed it. I said ``gives a presentation'' instead, and changed the beginning of the following sentence to ``More precisely'' rather than ``in other words''.}
The intersection form on $Z$ gives a presentation of the restriction of the linking form on $L(p,p-2)$ to a subgroup $G$, where $G$ is as in Proposition~\ref{p:MY}; see e.g.~\cite[Lemma~E.1]{GilmerLivingston_11}. More precisely, if the intersection form of $Z$ is isomorphic to $\langle \pm r \rangle$ for some positive integer $r$, then the linking form of $L(p,p-2)$ restricted to
$G$ is isomorphic to $\langle \mp \frac{1}{r} \rangle$.
Recall that the linking form $\lambda$ of $L(p,p-2)$ is isomorphic to $\langle \frac{p-2}p \rangle = \langle \frac{-2}p \rangle$. Thus, $\lambda|_G$ presents $\frac{-2}{r}$ as a square, in addition to presenting $\frac{1}{r}$ or $\frac{-1}{r}$ as a square.

Note that, since all odd squares are congruent to 1 modulo 8, $r \equiv p \pmod 8$.
If $p \equiv 7 \pmod 8$,
$r$ is not a square and %the intersection form of $Z$ is $\langle \pm r \rangle$ for some $r \equiv 7 \pmod 8$. In particular,
it is either divisible by a prime congruent to $7$ modulo $8$, or by a prime congruent $5$ modulo $8$. If $p \equiv 5 \pmod 8$, again $r$ is not a square and %the intersection form of $Z$ is $\langle \pm r \rangle$ for some $r \equiv 5 \pmod 8$. In particular,
it is either divisible by a prime congruent to $5$ modulo $8$ or by both a prime congruent $3$ modulo $8$ and $7$ modulo $8$.

We can exclude the case where $r=|G|$ is divisible by a prime congruent to $5$ modulo $8$ or by both a prime congruent $3$ modulo $8$ and one congruent $7$ modulo $8$ since otherwise, arguing as in the proof of Proposition~\ref{p:T45}, we find that $\lambda|_G$ does not present $\frac{-2}{r}$ as square, which yields a contradiction. In particular, this concludes the proof if $p \equiv 5 \pmod 8$.

It remains to treat the case where $p \equiv 7 \pmod 8$ with $p>7$ and $r=|G|$ is divisible by a prime $s \equiv 7 \pmod 8$ but not by a prime congruent to $3$ modulo $8$.
Since $s\equiv 7 \pmod 8$, %arguing as in the proof of Proposition~\ref{p:T45}, we one can use
Lemma~\ref{l:presentingassquares}\eqref{linking3mod4} implies that the reduction of $\lambda|_G$ modulo $s$ cannot present $\frac1s$ as a square, as this would contradict the reduction of $\lambda|_G$ modulo $s$ representing $-\frac{2}{s}$ as a square. Hence $\lambda|_G$ cannot present $\frac1r$ as a square. Therefore, the linking form restricted to $G$ is isomorphic to
$\langle - \frac{1}{r} \rangle$ (and not $\langle \frac{1}{r} \rangle$).
It follows that $Z$ is positive definite, i.e.~that its intersection form is isomorphic to $\langle r \rangle$.

At this point, we further assume that $M$ is smooth, hence $Z$ is smooth.
Gluing $-Z$ and the negative definite plumbing $P$ whose boundary is $L(p,p-2)$, we obtain a smooth, negative definite $4$--manifold $X$. In particular, the intersection form of $P$ embeds with co-rank $1$ in a diagonal lattice, by Donaldson's diagonalisation theorem~\cite{Donaldson}.

Recall that the negative definite plumbing of $L(p,p-2)$ is determined by the negative continued fraction expansion of $p/(p-2) = [2,\dots,2,3]$, where the string of $2$s has length $N-2 = (p-3)/2$. However, if $p > 7$ the string of $2$s has a unique embedding, namely (in some basis $e_1,\dots,e_N$ of $\Z^N$)
$e_1-e_2,\dots,e_{N-2} - e_{N-1}$.
Since $N \ge 2$ by assumption, there is no vector in $\Z^N$ of self-intersection $-3$ that intersects the chain (algebraically) once in the last vector, which contradicts our assumption.
%To prove the assertion about locally-flat M\"obius bands bounding $K_{15/13}$, we observe that the previous argument gives an embedding of the intersection form of $P$ into a negative definite lattice of rank $8$. There are only two such lattices, namely $\Z^8$ and $E_8$. We proved above that $P$ does not embedding into $\Z^n$; since $P$ contains a class of square $-3$, it cannot embed into the even lattice $E_8$ either. Therefore, $K_{15/13}$ does not bound a locally-flat M\"obius band.
\end{proof}

\bibliography{squarepeg}
\bibliographystyle{alpha}

\vskip 0,3cm

\end{document}